\documentclass[11pt]{amsart}
\usepackage{amssymb,amsfonts,amsxtra,amsmath,tikz-cd,xcolor,mathrsfs,verbatim,centernot,enumitem,mathtools}
\usepackage{quiver}
\usepackage{relsize}    
\usepackage{nicematrix}
\usepackage[all]{xy}
\usepackage{dsfont}
\usetikzlibrary{decorations.markings}
\tikzset{negated/.style={
        decoration={markings,
            mark= at position 0.5 with {
                \node[transform shape] (tempnode) {$\backslash$};
            }
        },
        postaction={decorate}
    }
}

\def\@firstoftwo@second#1#2{%
  \def\temp##1.##2\@nil{##2}%
   \temp#1\@nil}
\newcommand\sref[1]{%
   \expandafter\@setref\csname r@#1\endcsname\@firstoftwo@second{#1}%
}

\usepackage{calligra,mathrsfs}
\DeclareMathOperator{\Hom}{\mathscr{H}\text{\kern -3pt {\calligra\large om}}\,}
\DeclareMathOperator{\Diff}{\mathscr{D}\text{\kern -3pt {\calligra  iff }}\,}
\newtheorem{theo}{Theorem}[section] 
\newtheorem*{theo*}{Theorem} 
\newtheorem{lemm}[theo]{Lemma}

\newtheorem*{conj*}{Conjecture} 
\newtheorem{coro}[theo]{Corollary}

\newtheorem{prop}[theo]{Proposition}
\newtheorem{rema}[theo]{Remark}

\theoremstyle{definition}
\newtheorem{setu}[theo]{Setup}
\newtheorem{exam}[theo]{Example}
\newtheorem*{exam*}{Example} 
\newtheorem{defi}[theo]{Definition}

\newcommand{\CC}{\mathbb{C}}

\newcommand{\PP}{\mathbb{P}}

\newcommand{\RR}{\mathbb{R}}
\newcommand{\ZZ}{\mathbb{Z}}

\newcommand{\cI}{\mathcal{I}}

\newcommand{\cB}{\mathcal{B}}

\newcommand{\cL}{\mathcal{L}}

\newcommand{\cU}{\mathcal{U}}

\newcommand{\cR}{\mathcal{R}}
\newcommand{\cP}{\mathcal{P}}
\newcommand{\cQ}{\mathcal{Q}}
\newcommand{\cJ}{\mathcal{J}}
\newcommand{\cF}{\mathcal{F}}

\newcommand{\sO}{\mathscr{O}}
\newcommand{\sP}{\mathscr{P}}

\newcommand{\rk}{\operatorname{rk}}

\def\and{\quad\mathrm{and}\quad}

\def\Hom{{\textup{Hom}}}

\def\beq{\begin{equation}}
\def\eeq{\end{equation}}

\begin{document}
\title[Syzygy Rank Characterization of Strongly Euler Homogeneity]{A Syzygy Rank Characterization of Strongly Euler Homogeneity for Projective Hypersurfaces}

\author{Xia Liao}
\email{xliao@hqu.edu.cn}
\address{Department of Mathematical Sciences, 
Huaqiao University, 
Chenghua North Road 269, 
Quanzhou, Fujian, China}

\author{Xiping Zhang}
\email{xzhmath@gmail.com}
\address{
School of Mathematical Sciences, Key Laboratory of Intelligent Computing and Applications (Ministry of Education),
Tongji University, 
1239 Siping Road, 
Shanghai, China
}


\begin{abstract} 
In this paper we give a characterization of strongly Euler homogeneous singular points on a reduced complex projective hypersurface  $D=V(f)\subset \PP^n$ using the  Jacobian syzygies  of $f$. The characterization compares the ranks of the first syzygy matrices of the global Jacobian ideal $J_f$ and its quotient $J_f/(f)$.  
When $D$ has only isolated singularities, our characterization refines a recent result of Andrade-Beorchia-Dimca-Mir\'{o}-Roig. We also prove a generalization of this characterization to smooth projective toric varieties.
\end{abstract}
\maketitle

\section{Introduction}

Let $g:(\CC^n,0) \to (\CC,0)$ be a complex analytic singularity. This singularity is quasi-homogeneous if one can express $g$ as a weighted homogeneous polynomial (with strictly positive weights) in certain holomorphic coordinate system $x_1,\ldots, x_n$ of $\CC^n$. Quasi-homogeneous isolated hypersurface singularity is a considerably well-studied type of singualrities in algebraic geoemtry. A relevant but less well-known singularity type is strongly Euler homogeneous singularity. Here we recall its definition.  

\begin{defi}
\label{defn; SEH}
If in some holomorphic local coordinates of $(\CC^n,0)$ we have $g=\sum_{i=1}^n a_i \frac{\partial g}{\partial x_i}$ with $a_i\in \mathfrak{m}_{\CC^n,0}$ for every $i$, then we say the function germ $g$ is strongly Euler homogeneous. Let $D$ be a reduced hypersurface on a complex manifold $M$. We say that $D$ is strongly Euler homogeneous at a point $p\in D$ if $D$ is defined by a strongly Euler homogeneous holomorphic function germ in some analytic neighborhood of $p$. 

\end{defi}

Note that quasi-homogeneous singularities are strongly Euler homogeneous, but not vice versa. If the hypersurface  singularity is isolated, a celebrated result of K.Saito (see \cite{KSaito71}) states that the singularity is strongly Euler homogeneous if and only if it is quasi-homogeneous. Recently, the characterisation of quasi-homogeneity and the strong Euler homogeneity of hypersurface singularities have received some renewed interest (see \cite{ABDM25},\cite{ABM25},\cite{Rodriguez25},\cite{YZ17}).

Let $f\in R:=\CC[x_0, x_1, \cdots ,x_n]$ be a degree $d$ homogeneous polynomial that defines a reduced hypsrsurface $D=V(f)\subset \PP^n$. Under the assumption that $D$ has only isolated singularities, it is shown in \cite{ABDM25} that the rank of the first syzygy matrix of the global Jacobian ideal $J_f$ (see Definition \ref{defi:syzygyrank} below) can be used to characterize the quasi-homogeneity of $D$ at its singular points. With the aid of a computer algebraic system, this result can help to find the quasi-homogeneous isolated singularities on a projective hypersurface very effectively.

The simple observation we make in this paper is that, the main results in \cite{ABDM25} and \cite{ABM25} can be derived directly from a special case of Rodr\'{i}guez's characterization of strong Euler homogeneity in \cite[Proposition 2.6]{Rodriguez25}. An equivalent geometric characterization of strong Euler homogeneity using the language of log transversality when $D$ is a free divisor is given in \cite[Proposition 5.10]{LZ24-1}. If we employ the full statement of Rodr\'{i}guez's theorem, we can obtain a characterization of strong Euler homogeneity of any projective hypersurface using syzygy ranks. 

To be more precise, let $J_f=(f_{x_0}, \cdots , f_{x_n})$ be the global Jacobian ideal. It admits a free resolution of the following form: 
\[
\begin{tikzcd}
\cdots \arrow[r, ] & \bigoplus_{i=1}^m R(-d_i) \arrow[r, "M'_f"] & R^{n+1} \arrow[r, ""] & J_f(d-1) \arrow[r, ] & 0.
\end{tikzcd}
\]

Here we refrain from talking about the minimal free resolution of $J_f$ because the rank of its zeroth syzygy may be less than ${n+1}$, e.g. $f$ misses some variables. Since $f$ is homogeneous we have $\sum_{i=0}^n x_i f_{x_i} =d\cdot f$ so that $f\in J_f$. The 
above free resolution  induces the following   free resolution of $\left(J_f/(f)\right)(d-1)$:
\[
\begin{tikzcd}
\cdots \arrow[r, ] & \bigoplus_{i=0}^m R(-d_i)  \arrow[r, "M_f"] & 
 R^{n+1} \arrow[r, ""] & \left(J_f/(f)\right)(d-1) \arrow[r, ] & 0
\end{tikzcd} 
\]
where $d_0=1$ and the generator of $R(-d_0)$ is sent to  $(x_0,x_1, \cdots ,x_n)$ by $M_f$.

Given any $p=[p_0:p_1:\cdots : p_n]\in D$,  we may evaluate the  matrices $M'_f(\lambda p)$ and $M_f(\lambda p)$ 
at any point $\lambda p:=(\lambda p_0, \lambda p_1, \cdots , \lambda p_n)$  in the affine cone $\hat{D}$. The scalar matrices $M'_f(\lambda p)$ and $M_f(\lambda p)$   changes along $\lambda$, but their ranks are independent of $\lambda\neq 0$. We denote them  by $\rk M'_f(p)$ and $\rk M_f(p)$.  
\begin{defi}
\label{defi:syzygyrank}
We call the     matrix $M'_f$  the first syzygy matrix of $J_f$ and call the    matrix $M_f$  the first augmented syzygy matrix of $J_f$.  
For any point $p\in D$, we call  $\rk M'_f(p)$ the syzygy rank of $f$ at $p$ and call  $\rk M_f(p)$ the augmented syzygy rank of $f$ at $p$. 
\end{defi}

A direct translation of Rodr\'{i}guez's theorem from local analytic setting to global projective setting yields the following result, proved in Theorem~\ref{theo; syzygyrankcriterion}. 
\begin{theo}[Syzygy Rank Criterion]
The projective hypersurface $D$ is strongly Euler homogeneous at a point $p\in D$ if and only if 
\[
\rk M'_f(p)  =\rk M_f(p)  \/.
\]
 \end{theo}

 In \S\ref{sec; toric} we generalize this theorem to    smooth projetcive toric varieties. 
Let $X$ be a smooth projetcive toric variety and $D\subset X$ be a reduced hypersurface cut by a global function $f$.
Similarly we define global Jacobian  syzygies and introduce matrices $M'_f$ and  $M_f$ in \eqref{seq; toricsyzygy1} and \eqref{seq; toricsyzygy2}.
The major difference in the toric setting is  the logarithmic defect $\textup{Def}_{X,f}(p)$ we introduce  in Definition~\ref{defi; logdefect}. This is a non-negative integer and equals $0$ when  $X=\PP^n$ and $D$ is  any reduced hypersurface. In Theorem~\ref{theo; toricsyzygy} we prove     that
$D$ is strongly Euler homogeneous at a point $p\in D$ if and only if 
\[
\rk M'_f(p) + \textup{Def}_{X,f}(p)=\rk M_f(p)  \/.
\]

It can be shown that $\rk M_f(p)=\rk \textup{Pic}(X)$, the Picard number of $X$ whenever $p$ is an isolated singular point of $D$.  
For projective hypersurfaces,  since $\rk M'_f(p) \leq \rk M_f(p)=1$, in Corollary~\ref{coro; isohyper} we give a refined version of  \cite[Theorem 1.1]{ABDM25}.
\begin{coro}
Let $p$ be an isolated singular point of a projective hypersurface $D$, then  
\begin{enumerate}  
\item  $\rk M'_f(p)=1$ if and only if   $D$ is  quasi-homogeneous at $p$. 
\item  $\rk M'_f(p)=0$ if and only if   $D$ is  not quasi-homogeneous at $p$. 
\end{enumerate}   
\end{coro}

 For planar curves in $\PP^2$, the matrix $M_f$ defines a subscheme $Z_f$ of $\mathbb{P}^2\times \mathbb{P}^2$ (see \S\ref{section:SEHandlogCC}). It is proved in \cite[Theorem 4.7]{ABDM25} (see also Proposition \ref{prop:ABDMrestate}) that the quasi-homogeneity of a curve in $\mathbb{P}^2$ is equivalent to the irreducibility of $Z_f$. We point out in \S\ref{section:SEHandlogCC} that this is not an exclusive phenomenon for curves in $\mathbb{P}^2$, but it follows from some general results (see Corollary~\ref{coro; CC}) relating the strong Euler homogeneity of a reduced divisor $D$ in $\PP^n$ to properties of log characteristic cycles.  
 Throughout \S\ref{section:SEHandlogCC} we assume that the reader is familiar with the basic calculus of constructible functions and characteristic cycles.

\vspace{0.5cm}
\noindent
{\bf Acknowledgements}: 
The second author would like to thank Botong Wang, Zhenjian Wang and Zhixian Zhu for helpful  discussions. 
Xia Liao is supported by Chinese National Science Foundation of China (Grant No.11901214).
Xiping Zhang is supported by National Natural Science Foundation of China (Grant No.12201463).

\section{Logarithmic Derivations and the (augmented) Logarithmic Rank}
\label{sec; logrank}

Let $X$ be a complex manifold of dimension $n$ and $L$ be a line bundle on $X$. 
Let  $f\in H^0(X,L)$ be a non-trivial global section and $D=V(f)$ be a reduced hypersurface cut by $f$. Denote the sheaf of sections of $L$ by $\mathcal{L}$.

\begin{defi}
\label{defn; logDer}
The sheaf of logarithmic derivations along $D$, denoted by $\textup{Der}_X(-\log D)$,  is locally given by 
\begin{equation*}
\textup{Der}_X(-\log D)(U)=\{\chi \in \textup{Der}_X(U) \ | \  \chi(h) \in h\mathscr{O}_U\}
\end{equation*}
for any open subset $U\subset X$ and any local representative  $h$ of $f$ on $U$.   
\end{defi} 
This sheaf was  introduced  by Saito  in \cite{MR586450},  where he proved that it is a coherent and reflexive $\sO_X$-module, and has generic  rank $n$ outside $D$.

Let $\cI_D$ be the ideal sheaf of $D$ and 
$\cJ_D$ be the ideal sheaf of the singular subscheme of $D$. Clearly $\cI_D\subset \cJ_D$ and let the $\sO_X$-module  
$\cR_D:=\cJ_D/\cI_D$ be the quotient of the ideal sheaves. 
By definition of $\textup{Der}_X(-\log D)$, we have the following short exact sequence
\begin{equation}
\label{seq; sesq1global}
\begin{tikzcd}
0 \arrow[r] & \textup{Der}_X(-\log D)\otimes \mathcal{L}^\vee \arrow[r, " "] & \textup{Der}_X \arrow[r, "\rho"]\otimes \mathcal{L}^\vee & \cR_D  \arrow[r]  & 0.
\end{tikzcd}
\end{equation}

Next, we introduce another related short exact sequence. 
Let $P_X^1L$ be the bundle of principal parts of $L$. It is a rank $n+1$ vector bundle on $X$ associated to a locally free sheaf $\mathcal{P}^1_X\mathcal{L}$ whose precise definition can be found in \cite[\S 4]{Atiyah57} and \cite[Section 16]{EGAIV-4}. On any open subset $U$ where both $L$ and $T^*X$ are trivialized, we have the following concrete description of $\mathcal{P}^1_X\mathcal{L}$: 
\[
\Gamma(U, \mathcal{P}^1_X\mathcal{L})  = \left\lbrace j^1(g):=(g_{x_1}, g_{x_2}, \cdots , g_{x_n}, g)\in \mathscr{O}_{U}^{n+1}| g\in \mathscr{O}_U \cong \Gamma(U, L) \right\rbrace.
\]   
The global section $j^1(f) \in H^0(X,P^1_XL)$ corresponds to a morphism of sheaves  $\rho'\colon \left(\mathcal{P}^1_X\mathcal{L} \right)^\vee \to \sO_X$.  A local computation shows that the image of $\rho'$ is  $\cJ_D\otimes \mathcal{L}$ and the kernel is the inclusion  $\textup{Der}_X(-\log D)\otimes L^\vee \hookrightarrow (\mathcal{P}^1_X\mathcal{L})^\vee$ (cf. \cite[\S 3]{LZ24-1} for a dual description of this inclusion). We then have the following short exact sequence 
\begin{equation}
\label{seq; sesq2global}
\begin{tikzcd}
0 \arrow[r] & \textup{Der}_X(-\log D)\otimes L^\vee \arrow[r, " "] & (\mathcal{P}_X^1\mathcal{L})^\vee \arrow[r, "\rho'"] & \cJ_D \arrow[r]  & 0.
\end{tikzcd}
\end{equation}

Given  any  surjective morphism $q: \sP_1 \to \textup{Der}_X(-\log D)\otimes \mathcal{L}^\vee$ from a coherent locally free $\sO_X$-module $\sP_1$, we obtain two exact sequences of $\sO_X$-modules.
\begin{equation}
\label{seq; j}
\begin{tikzcd}
\sP_1 \arrow[r, "j^\vee"]  &  \textup{Der}_X \otimes \mathcal{L}^\vee \arrow[r, "\rho"] & \cR_D \arrow[r]  & 0  \/.
\end{tikzcd}  
\end{equation}
\begin{equation}
\label{seq; j'}
\begin{tikzcd}
\sP_1 \arrow[r, "j'^\vee"]  &  (\mathcal{P}_X^1\mathcal{L})^\vee\arrow[r, "\rho'"] & \cJ_D \arrow[r]  & 0.  
\end{tikzcd}
\end{equation}

The superscript $\vee$ is put here for $j$ and $j'$ to conform with notations in \cite{LZ24-1}. Let $E=\textup{Spec}\left(\textup{Sym}^\bullet \sP_1 \right)$ be the vector bundle whose sheaf of sections is $\sP_1^\vee$. Evaluating at any $p\in D$ we obtain linear transformations
\[
j(p)\colon  T_p^*X\otimes L_p \to  E_p \ \text{ and } \  j'(p)\colon (P_X^1L)_p \to E_p.
\]

It is clear that the ranks of the linear maps $j(p)$ and $j'(p)$ are independent of all choices. 
\begin{defi}
\label{defi; logrank}
We call the  germs  $j_p$ and $j'_p$   the logarithmic morphism and the augmented logarithmic morphism at $p$ respectively. We also call $\rk j(p)$ and $\rk j'(p)$  the logarithmic rank and the augmented logarithmic rank of $D$ at $p$ respectively.
\end{defi}

With the additional geometric insights explained in \cite[\S 3]{LZ24-1}, Rodr\'{i}guez's result \cite[Proposition 2.6]{Rodriguez25} can be interpreted in the following form. 
\begin{prop} 
\label{prop; logrankSEH}
The hypersurface $D$ is strongly Euler homogeneous at $p$ if and only if 
\[
\rk j'(p) =\rk j(p) +1 \/.
\]
\end{prop}
 
Now we give a description of the above constructions in local coordinates. Suppose $\cU$ is a small enough analytic neighborhood of $p$. Let  $\underline{x}=(x_1, \cdots ,x_n)$ be the  local coordinates in $\cU$ and 
$h\in \sO_\cU$ be a local defining equation of $D\cap \cU$. Then we have
\[
\cI_D(\cU)=(h)\subset \cJ_D(\cU)=(h_{x_1}, h_{x_2}, \cdots ,h_{x_n}, h)\subset  \sO_\cU \ \text{ and } \ \cR_D(\cU)=\cJ_D(\cU)/(h) \/.
\]
The morphism $\rho$ in sequence~\eqref{seq; sesq1global} and the morphism $\rho'$ in Sequence~\eqref{seq; sesq2global} are given by 
\begin{align*}
\rho_\cU \colon  \textup{Der}_\cU\cong \sO_\cU^{\oplus n}\to \cJ_D(\cU)/(h)
\/, & \quad (a_1, \cdots ,a_n)\mapsto \sum_{i=1}^n a_i \overline{h}_{x_i}  \/,\\  
\rho'_\cU\colon (\cP_X^1\cL)^\vee\cong\sO_\cU^{\oplus n+1}\to \cJ_D(\cU)
\/, & \quad (a_1, \cdots ,a_n, b)\mapsto \sum_{i=1}^n a_i h_{x_i} +bh \/.
\end{align*}
It's straightforward to verify that 
$
\textup{Der}_\cU(-\log D)=\ker \rho_\cU=\ker \rho'_\cU \/.
$

We also identify $\sP_1|_\cU$ with $\sO_\cU^{\oplus m}=\oplus_{i=1}^m \sO_\cU\cdot e_i$ and define 
\[
\cB:=\left\lbrace \beta_k:=q(e_k)= \sum_{i=1}^n \beta_{k,i}\partial_{x_i}  \in \textup{Der}_\cU(-\log D)\/; \quad k=1,2, \cdots ,m \right\rbrace \/.
\]
Since $q$ is surjective, $\cB$ forms  a generating set of $\textup{Der}_{\cU}(-\log D)$. Under this generating set we may write $j^\vee_p$ and $j'^\vee_p$ by 
 \begin{equation}
\tag{$\dagger$}\label{eq; matrixrepnofj}
j^\vee_p= \begin{bmatrix}
\beta_{1,1} & \beta_{2,1} &\cdots  & \beta_{m,1} \\
\beta_{1,2} & \beta_{2,2} &\cdots  & \beta_{m,2} \\
\vdots & \vdots & \cdots & \vdots \\
\beta_{1,n} & \beta_{s,2} &\cdots  & \beta_{m,n}  
\end{bmatrix}
\ \text{ and }\ 
j'^\vee_p= \begin{bmatrix}
\beta_{1,1} & \beta_{2,1} &\cdots  & \beta_{m,1} \\
\beta_{1,2} & \beta_{2,2} &\cdots  & \beta_{m,2} \\
\vdots & \vdots & \cdots & \vdots \\
\beta_{1,n} & \beta_{s,2} &\cdots  & \beta_{m,n}  \\
g_1 & g_2 & \cdots & g_m
\end{bmatrix} \/,
\end{equation}
where $\{g_1,\ldots,g_m\}\subset \sO_\cU$ are given by $\beta_k(h)=g_kh$.  

We finish this section with the following decomposition property:

\begin{prop}
\label{prop; decompofP1L}
Let $X=\PP^n$ be the projective space and $\mathcal{L}=\sO_{\PP^n}(k)$ for some $k\geq 1$. Then the bundle of principal parts splits.
\[
\cP^1_{\PP^n}\sO_{\PP^n}(k) \cong \sO_{\PP^n}(k-1)^{n+1} \/.
\] 
\end{prop} 

This isomorphism is explained in \cite[p37]{MR1420708} as a part of a commutative diagram involving bundles of principal parts on smooth toric varieties. We will provide an alternative proof below.

\begin{proof}
Recall that we always have $\mathcal{P}^1_X\mathcal{L}\otimes\mathcal{L}^\vee \cong \Omega^1_L(\log X)\vert_X$ by \cite[Corollary 3.18]{LZ24-1} where $X$ is any complex manifold regarded as the zero section of $L$. So in the case that $\mathcal{L}\cong \mathscr{O}_{\mathbb{P}^n}(k)$, we only need to show $\Omega^1_L(\log \mathbb{P}^n)  \vert_{\mathbb{P}^n} \cong \mathscr{O}_{\mathbb{P}^n}(-1)^{n+1}$.

Let $s_i=x_i^k\in H^0(\mathbb{P}^n,\mathscr{O}(k))$. So $s_i$ trivializes $L$ on the standard open set $U_i$. The relation between $s_i$ and $s_j$ on $U_i\cap U_j$ is $s_i\cdot (\frac{x_j}{x_i})^k=s_j$. The section $s_i$ determines a coordinate function $t_i \in \mathscr{O}(U_i)$ for the fibres of $L$ on $U_i$. The relation between $t_i$ and $t_j$ on $U_i\cap U_j$ is $t_i \cdot (\frac{x_i}{x_j})^k=t_j$. From here we deduce that 
\[
\frac{1}{k}\frac{dt_j}{t_j}-\frac{d(\frac{x_i}{x_j})}{\frac{x_i}{x_j}}=\frac{1}{k}\frac{dt_i}{t_i}
\]
on $U_i\cap U_j$.

Note that $\frac{1}{k}\frac{dt_i}{t_i}$ is extended to a global rational section of $\Omega^1_L(\log \mathbb{P}^n)  \vert_{\mathbb{P}^n}$. Let's name this rational section by $r_i$. The equation above shows that $r_0\otimes x_0,\ldots,r_n\otimes x_n$ are nonvanishing global section of $\Omega^1_L(\log \mathbb{P}^n)  \vert_{\mathbb{P}^n} \otimes \mathscr{O}(1)$. For example, let's examine $r_i\otimes x_i$. On $U_i$, it is trivialized into $\frac{1}{k}\frac{dt_i}{t_i}$ so it is regular and nonvanishing. On $U_j$, it is trivialized into 
\[
(\frac{1}{k}\frac{dt_j}{t_j}-\frac{d(\frac{x_i}{x_j})}{\frac{x_i}{x_j}})\otimes \frac{x_i}{x_j} = \frac{1}{k}\frac{dt_j}{t_j}\otimes \frac{x_i}{x_j} - d(\frac{x_i}{x_j})\otimes 1
\]
which is also regular and nonvanishing because $d(\frac{x_i}{x_j})$ appears. It is clear that $r_0\otimes x_0, \ldots, r_n\otimes x_n$ are pointwise linearly independent, hence $\Omega^1_L(\log \mathbb{P}^n)  \vert_{\mathbb{P}^n} \otimes \mathscr{O}(1) \cong \mathscr{O}^{n+1}$.
\end{proof}

\section{Syzygy Rank of Projective Hypersurfaces} 
\label{sec; syzygyrank}
In  this section we focus on reduced   projective hypersurfaces. Let $R=\CC[x_0, x_1, \cdots ,x_n]=\bigoplus_{j\in \ZZ} R_j$ be the polynomial ring with the standard grading: $R_j$ 
denotes the vector space of all homogeneous polynomial of degree $j$. 

A reduced degree $d$ hypersurface $D\subset \PP^n$ is cut by a homogeneous polynomial $f\in R_d$. We denote by $\hat{D}$ the affine cone of $D$ in $\CC^{n+1}$. 
Let $J_f=(f_{x_0}, f_{x_1}, \cdots ,f_{x_n})$ be the global Jacobian ideal of $f$ and let $K$ be the kernel of $Df: R^{n+1} \to J_f(d-1)$ sending the ith generator to $f_{x_i}$. Suppose a mimimal generating set of $K$ has $m$ relations, then we have an exact sequence
\begin{equation}
\label{seq; M_f}
\begin{tikzcd}
\bigoplus_{i=1}^m R(-d_i) \arrow[r, "M'_f"] & R^{n+1} \arrow[r,"Df"] & J_f(d-1) \arrow[r,] & 0.
\end{tikzcd}
\end{equation}

The graded map $M'_f$ is represented by a  matrix with entries in $R$
\begin{equation}
\tag{$\ddagger$}\label{eq; presentationofM_f}
M'_f=
\begin{bmatrix}
\delta_{1,0} & \delta_{2,0} & \cdots & \delta_{m,0} \\
\vdots & \vdots & \cdots & \vdots \\
\delta_{1,n} & \delta_{2,n} & \cdots & \delta_{m,n}  
\end{bmatrix}
\end{equation}
where each column gives a homogeneous annihilating relation of $J_f$, i.e., 
\[
\sum_{i=0}^n \delta_{k,i} f_{x_i}  =0  \text{ and } 
\{\delta_{k,0}, \delta_{k,1}, \cdots , \delta_{k,n}\}\subset  R_{d_k} \/, 
\quad  \forall   k=1,2, \cdots , m  \/.
\]

By construction, the following proposition is clear.
\begin{rema}
The following vector fields   generate the $R$-module $\textup{Der}_{\CC^{n+1}}(-\log \hat{D})$:
\[
\cB_f:=\left\lbrace
\chi=\sum_{i=0}^n x_i\partial_{x_i}, \ \delta_1:=\sum_{i=0}^n \delta_{1,i}\partial_{x_i}, \  \delta_2:=\sum_{i=0}^n \delta_{2,i}\partial_{x_i} ,\  \cdots  \ , \ \delta_m:=\sum_{i=0}^n \delta_{m,i}\partial_{x_i}
\right\rbrace \/.
\]
\end{rema}

\begin{proof}
Let $\delta$ be any derivation satisfying $\delta(f)=f$, then we can split $\delta$ into $\delta=\frac{1}{d}\chi+(\delta-\frac{1}{d}\chi)$. Since the second term annihilates $f$, it is generated by $\delta_1,\ldots,\delta_m$.
\end{proof}

Consider the induced free resolution of $\left(J_f/(f)\right)(d-1)$:
\begin{equation*}
\begin{tikzcd}
\bigoplus_{i=0}^m R(-d_i)  \arrow[r, "M_f"] & 
 R^{n+1} \arrow[r, ""] & \left(J_f/(f)\right)(d-1) \arrow[r, ] & 0
\end{tikzcd}
\end{equation*}
where $d_0=1$ and  $R(-d_0)$ represents the relation $\sum_{i=0}^n x_i f_{x_i} =df$. Then $M_f$ is represented by
$$
M_f= 
\begin{bmatrix}
\delta_{1,0} & \delta_{2,0} & \cdots & \delta_{m,0} & x_0\\
\vdots & \vdots & \cdots & \vdots & \vdots \\
\delta_{1,n} & \delta_{2,n} & \cdots & \delta_{m,n} & x_n
\end{bmatrix} \/.
$$

Now we show that the syzygy rank equals the augmented logarithmic rank. 
\begin{lemm}
\label{lemm; syzygyrank}
For any $p\in D$ we have
$
\rk j'(p)=\rk M'_f(p) \/.
$
\end{lemm} 
\begin{proof}
Recall that (see \cite[Chapter 5]{MR1810311}) the category $QCoh(\PP^n)$ of coherent sheaves on $\PP^n$ and the category $Gr_fMod(R)$ of finitely generated graded $R$-modules are connected by 
two functors
\[
 \Gamma_*\colon QCoh(\PP^n) \to Gr_fMod(R), \quad  \beta:=\widetilde{(\bullet)}\colon Gr_fMod(R)  \to QCoh(\PP^n)\/,
\]
such that $\beta\circ \Gamma_*$ is an isomorphism and $\beta$ is exact.

Applying $\beta$ to the short exact sequence
\[
\begin{tikzcd}
0 \arrow[r]& K(1-d)  \arrow[r] &  R(1-d)^{n+1} \arrow[r, "Df"] 
&  J_f  \arrow[r] & 0 \/,
\end{tikzcd}
\]
we obtain
\[
\begin{tikzcd}
0 \arrow[r] & \beta(K)(1-d) \arrow[r] & \sO_{\PP^n}^{n+1}(1-d)  \arrow[r] 
& \cJ_D  \arrow[r] & 0.
\end{tikzcd}
\]
By short exact sequence~\eqref{seq; sesq2global} and Proposition~\ref{prop; decompofP1L}, we conclude that 
\[
\beta(K)(1-d)\cong  \textup{Der}_{\PP^n}(-\log D)\otimes \mathcal{L}^\vee.
\]
Applying $\beta$ to the surjection $\bigoplus_{i=1}^m R(-d_i) \to K$ gives rise to a surjective morphism
\[
\begin{tikzcd}
\bigoplus_{i=1}^m \sO_{\PP^n}(1-d-d_i) \arrow[r] & \beta(K)(1-d) \arrow[r] & 0.
\end{tikzcd}
\]
Therefore we can let $\bigoplus_{i=1}^m \sO_{\PP^n}(1-d-d_i)$ be $\mathscr{P}_1$. We deduce that the combination
\[
\begin{tikzcd}
\bigoplus_{i=1}^m \sO_{\PP^n}(1-d-d_i) \arrow[r] & \beta(K)(1-d) \arrow[r] & \sO_{\PP^n}^{n+1}(1-d)
\end{tikzcd}
\]
is $j'^\vee$. In other words, we have proved $j'^\vee=\beta(M'_f(1-d))$.

\end{proof}

\begin{lemm}
\label{lemm; augsyzygyrank}
For any $p\in D$ we have
$
\rk j(p)+1=\rk M_f(p) \/.
$
\end{lemm}
\begin{proof}  
We define the graded $R$-module $N$ by the following short exact sequence 
\[
\begin{tikzcd}
0 \arrow[r]& R:=R\cdot (\sum_{i=0}^n x_i\partial_{x_i}) \arrow[r, " "] &   R(1)^{\oplus n+1}:=\bigoplus_{i=0}^n R\cdot \partial_{x_i}   \arrow[r, "\pi"] 
& N\arrow[r] 
&  0,
\end{tikzcd}
\]
then applying $\beta$ we have  $\beta(N)=\textup{Der}_{\PP^n}$ and   obtain the Euler sequence  on $\PP^n$. 

Quotienting out the middle term of sequence~\eqref{seq; M_f} by the free submodule generated by $(x_0,\ldots,x_n)$, we obtain
\[
\begin{tikzcd}
\bigoplus_{i=1}^m R(-d_i) \arrow[r, "T_f"] & N(-1) \arrow[r, ""] & \left(J_f/(f)\right)(d-1) \arrow[r, ] & 0
\end{tikzcd}
\]
Notice that $\beta \left(J_f/(f)\right) = \cR_D$, then applying $\beta$ to the above sequence we obtain
\[
\begin{tikzcd}
\bigoplus_{i=1}^m \sO_{\PP^n}(1-d_i) \arrow[r, "\beta(T_f(1))"] & \textup{Der}_{\PP^n}  \arrow[r, "\rho"] & \cR_D\otimes \mathcal{L} \arrow[r, ] & 0
\end{tikzcd} \/.
\]
Comparing with sequence~\eqref{seq; j} we then have $\beta(T_f(1-d))=j^\vee$.  

On the other hand, we have the  following commutative diagram
\[
\begin{tikzcd}
\bigoplus_{i=1}^m R(-d_i) \arrow[r, "T_f"] & N(-1) \arrow[r, " "] & \left(J_f/(f)\right)(d-1) \arrow[r, ] & 0 \\
\bigoplus_{i=0}^m R(-d_i)\arrow[u, "\pi'"]  \arrow[r, "M_f"] & R^{\oplus n+1} \arrow[u, "\pi"] \arrow[r, "Df"] &   \left(J_f/(f)\right)(d-1) \arrow[r, ] \arrow[u, "="] & 0   
\end{tikzcd}
\]
where $\pi'$ is the projection to the summands indexed by $1,\ldots,m$. 
This shows that $\rk M_f(p)=\rk T_f(p)+1$.
\end{proof}

\begin{rema}
Lemma~\ref{lemm; syzygyrank} and Lemma~\ref{lemm; augsyzygyrank} can also be seen directly in local coordinates. Without loss of generality  we may assume that $p=[1:0:\cdots : 0]$. Let 
$(y_i =\frac{x_i}{x_0})$ be the local coordinates around $p$ and $\delta_{k,j}$ be the polynomials  in the representation \eqref{eq; presentationofM_f} of $M'_f$. We define polynomials $g(y_1, \cdots ,y_n)$ and $\{\alpha_{k,j}(y_1, \cdots ,y_n)\}$ by 
\[
x_0^d\cdot g(\frac{x_1}{x_0},\cdots , \frac{x_n}{x_0})=f(x_0, \cdots, x_n)\/; \quad  x_0^{d_k}\cdot \alpha_{k,j}(\frac{x_1}{x_0},\cdots , \frac{x_n}{x_0})=\delta_{k,j}(x_0, \cdots, x_n) \/.
\]
We also define a set of local vector fields  
\[
\cB:=\left\lbrace \chi_k:=  \sum_{i=1}^n (\alpha_{k,i}- \alpha_{k,0}y_i)\partial_{y_i} \ \Big\vert k=1,2,\cdots ,m  \right\rbrace   \/.
\]
Now $g$ defines the hypersurface $D$ around $p$. A concrete computation shows that $\chi_k(g)=-d\alpha_{1,0}g$ and $\cB$
forms a local generating set of $\textup{Der}_{\PP^n}(-\log D)$. Then by the representation \eqref{eq; matrixrepnofj} we have:
\[
j^\vee_p  =
\begin{bmatrix}
\alpha_{1,1}- \alpha_{1,0}y_1 & \alpha_{2,1}- \alpha_{2,0}y_1 & \cdots & \alpha_{m,1}- \alpha_{m,0}y_1   \\
\alpha_{1,2}- \alpha_{1,0}y_2 & \alpha_{2,2}- \alpha_{2,0}y_2 & \cdots & \alpha_{m,2}- \alpha_{m,0}y_n   \\
\vdots & \vdots & \cdots  & \vdots  \\
\alpha_{1,n}- \alpha_{1,0}y_n  & \alpha_{2,n}- \alpha_{2,0}y_2  & \cdots & \alpha_{m,n}- \alpha_{m,0}y_n  
\end{bmatrix} \/,\quad 
j'^\vee_p =\begin{bNiceArray}{ccc}[margin]
    & j_p &   \\ 
    \hline 
-d\alpha_{1,0} &  \cdots  & -d\alpha_{m,0} 
\end{bNiceArray}  \/.
\]
Evaluating  at $p=(0,0,\cdots ,0)$  we have
\[
j^\vee(p)=\begin{bmatrix}
\alpha_{1,1}  & \alpha_{2,1}  & \cdots & \alpha_{m,1}   \\
\alpha_{1,2}  & \alpha_{2,2}  & \cdots & \alpha_{m,2}    \\
\vdots & \vdots & \cdots  & \vdots  \\
\alpha_{1,n}  & \alpha_{2,n}  & \cdots & \alpha_{m,n}   
\end{bmatrix}(0) 
\text{ and } 
j'^\vee(p)=\begin{bmatrix}
\alpha_{1,1}  & \alpha_{2,1}  & \cdots & \alpha_{m,1}   \\ 
\vdots & \vdots & \cdots  & \vdots  \\
\alpha_{1,n}  & \alpha_{2,n}  & \cdots & \alpha_{m,n}   \\
-d\alpha_{1,0} & -d\alpha_{2,0} & \cdots & -d\alpha_{m,0}
\end{bmatrix}(0) \/.
\]
Now the matrices $M'_f(p)$ and $M_f(p)$ are given by 
\[
M'_f(p)=\begin{bmatrix}
\alpha_{1,0} & \alpha_{2,0}   & \cdots & \alpha_{m,0}   \\
\alpha_{1,1} & \alpha_{2,1}    & \cdots & \alpha_{m,1}    \\
\vdots & \vdots & \cdots  & \vdots  \\
\alpha_{1,n} & \alpha_{2,n}     & \cdots & \alpha_{m,n}   
\end{bmatrix}(0) 
\text{ and } 
M_f(p)= 
\begin{bmatrix}
\alpha_{1,0} & \alpha_{2,0}   & \cdots & \alpha_{m,0} & 1  \\
\alpha_{1,1} & \alpha_{2,1}    & \cdots & \alpha_{m,1}  & 0  \\
\vdots & \vdots & \cdots  & \vdots & \vdots  \\
\alpha_{1,n} & \alpha_{2,n}     & \cdots & \alpha_{m,n}   & 0
\end{bmatrix}(0) \/.
\]
It's immediate that $\rk M'_f(p)=\rk j'(p)$ and $\rk M_f(p)=\rk j(p)+1$. 
\end{rema}

Combining Lemma~\ref{lemm; syzygyrank}, Lemma~\ref{lemm; augsyzygyrank} and
Proposition~\ref{prop; logrankSEH} we obtain the following syzygy rank characterization of strongly Euler homogeneity.
\begin{theo} 
\label{theo; syzygyrankcriterion}
Let $D\subset \PP^n$ be a reduced projective hypersurface cut by $f\in R_d$. 
Then $D$ is strongly Euler homogeneous at point $p\in D$ if and only if 
\[
\rk M'_f(p) =\rk M_f(p)  \/.
\]
\end{theo}

Recall that Saito’s results on weighted homogeneous polynomials (see \cite{KSaito71}) shows that, a hypersurface  germ with  isolated singularity is   strongly Euler homogeneous  if and only if it is quasi-homogeneous.  Thus the following corollary  refines \cite[Theorem 1.1]{ABDM25}.
\begin{coro}
\label{coro; isohyper}
Assume that  $D$ has at most isolated singularity at a point  $p\in D$. Then
\begin{enumerate} 
\item  $\rk M'_f(p)=n$ if and only if $p$ is   a smooth point of $D$.
\item  $\rk M'_f(p)=1$ if and only if $p$ is   singular   and $D$ is  strongly Euler homogeneous at $p$. 
\item  $\rk M'_f(p)=0$ if and only if $p$ is   singular   and $D$ is  not strongly Euler homogeneous at $p$. 
\end{enumerate}   
\end{coro} 
\begin{proof}
For any  point  $p\in D$   the rank
$\rk j(p)$ is equal to the dimension of the logarithmic stratum containing $p$. 
Since $D$ has only isolated singular points, the logarithmic strata of $D$   are either $D_{sm}$ or the singular points themselves. 
So $\rk j(p)=0$ when $p$ is an isolated singularity and $\rk j(p)=n-1$ when $p$ is smooth.
\end{proof}

\section{Characterizations of Strong Euler Homogeneity by Log Characteristic Cycles}
\label{section:SEHandlogCC}

Given any representation of $M'_f$ as in \eqref{eq; presentationofM_f}, we define a subscheme $Z_f\subset \PP^n \times \PP^n$ by the following homogeneous ideal
\[
 I_{Z_f}:= \left(\sum_{i=0}^n \delta_{1,i}(\underline{x})\cdot y_i,\  \sum_{i=0}^n \delta_{2,i}(\underline{x})\cdot y_i, \  \cdots \ ,  \  \sum_{i=0}^n \delta_{m,i}(\underline{x})\cdot y_i \right)
\]
where $\underline{x}=[x_0: x_1: \cdots :x_n]$ and $\underline{y}=[y_0: y_1: \cdots :y_n]$  are the homogeneous coordinates of the first and the second component respectively. We define a subvariety $S_f\subset \PP^n \times \PP^n$ as the closure of the graph of the polar map 
\[
\nabla f: \PP^ n  \to \PP^n \/; \quad 
 \underline{x}  \mapsto [\partial_0 f(\underline{x}):\ldots,:\partial_n f(\underline{x})].
\]
Finally we define an incidence variety 
\[
I:=V(Q)\subset \PP^n \times \PP^n \text{ where } Q:=x_0y_0+x_1y_1+\cdots +x_ny_n \/.
\]

When $n=2$, the following result is proved in \cite[Theorem 4.7]{ABDM25}.
\begin{prop}\label{prop:ABDMrestate}
$D$ is a quasi-homogeneous curve in $\PP^2$ if and only if $S_f =Z_f$. When $Z_f$ is irreducible there is an equality of cycles $[Z_f] = [S_f] + \sum m_i[P_i \times \mathbb{P}^2]$ where $\{ P_1,\ldots, P_s\}$ is the set of singularities on $D$. Moreover $\sum m_i= \mu(D)-\tau(D)$ where $\mu(D)$ is the global Milnor number and $\tau(D)$ is the global Tjurina number. 
\end{prop}

What happens is that this proposition fits into a general framework concerning log characteristic cycles along a reduced divisor on any complex manifolds. To explain the general construction, we consider a reduced divisor $D$ on a complex manifold $X$. We first recall some facts from \cite{MR3366865}.  

\begin{defi}  
If the natural morphism $\textup{Sym}_{\sO_{X,p}}(\cJ_{D,p}) \to \textup{Rees}_{\sO_{X,p}}(\cJ_{D,p})$ is an isomorphism, we say $D$ is of linear Jacobian type at $p\in D$ . We say $D$ is of linear Jacobian type if it is so at any $p\in D$.
\end{defi}

\begin{prop} 
If $D$ is of linear Jacobian type at $p$, then it is strongly Euler homogeneous at $p$. If $D$ is free and quasi-homogeneous at $p$, then $D$ is of linear Jacobian type at $p$.
\end{prop}

Let $\Lambda=\textup{CC}(\gamma)$ be the characteristic cycle of a constructible function $\gamma$ on $X$ and let $f\in H^0(X,L)$ be a nonzero holomorphic section defining $D$. In \cite{LZ24-1}, we generalized Ginzburg's local sharp construction and defined a cycle $\overline{J\Lambda^\sharp}$ in $P^1_XL$, and in case $D$ is a free divisor we defined a logarithmic characteristic cycle $J\Lambda^{\log}$. Let $j': P_X^1L \to T^*X(\log D)\otimes L$ be the map introduced in \S\ref{sec; logrank} (Dualize the morphism $\mathscr{P}_1 \to (\mathcal{P}_X^1\mathcal{L})^\vee$ in the sequence \eqref{seq; j'} and take $\mathscr{P}_1=\textup{Der}_X(-\log D)$ becaue we assume $D$ is free. This map was denoted by $j'_L$ in \cite{LZ24-1}.) and let $\gamma=\textup{Eu}^\vee_Z$ where $Z$ is an irreducible subvariety not contained in $D$. We showed $\overline{J\Lambda^\sharp}$ is an irreducible component of $j'^{-1}(J\Lambda^{\log})$, and other irreducible components all lie over singular locus of $D$ and their dimensions are no less than $n+1$ (\cite[Proposition 3.4, Remark 4.2, Theorem 5.13]{LZ24-1}). Let's emphasize that the global Ginzburg's sharp construction does not require the freeness of $D$, but other constructions in \cite{LZ24-1} were all carried out under the freeness assumption.

\begin{defi}[\cite{LZ24-1}]
We say $Z$ is log transverse to $D$ if $\overline{J\Lambda^\sharp}$ is the only irreducible component of $j'^{-1}(J\Lambda^{\log})$ (so that they agree as sets). We say $Z$ is weakly log transverse to $D$ if irreducible components of $j'^{-1}(J\Lambda^{\log})$ other than $\overline{J\Lambda^\sharp}$ are all contained in $T^*X\otimes L$.
\end{defi}

\begin{prop}[\cite{LZ24-1}]
Let $D$ be a free divisor. $X$ is weakly log transverse to $D$ if and only if $D$ is strongly Euler homogeneous.
\end{prop}

Since we will only consider the case $Z=X$ in the sequel (except in Remark \ref{remark:inverseLambdalog}), we will use the letter $\Lambda$ exclusively for $[X]=\textup{CC}(\textup{Eu}^\vee_X)=\textup{CC}((-1)^n1_X)$.
We showed in \cite[Proposition 5.1]{LZ24-1} that when $D$ is a free divisor with linear type Jacobian ideal, then $X$ is log transverse to $D$. In fact, in this case $j'^{-1}(X)$ even agrees with $\overline{J\Lambda^{\sharp}}$ as schemes where $X=J\Lambda^{\log}$ is the zero section of $T^*X(\log D)\otimes L$. 

Now we drop the freeness assumption and only assume $D$ is of linear Jacobian type. Examining the proof of Proposition 5.1 in \cite{LZ24-1}, it can be seen that the same proof works almost verbatim in the new case except that we replace the locally free sheaf $\textup{Der}_X(-\log D)\otimes L^\vee$ in the old proof by $\mathscr{P}_1$ (cf. the sequence \eqref{seq; j}). Regard $x$ as the zero section of the vector bundle $E=\textup{Spec}(\textup{Sym}^\bullet\mathscr{P}_1)$ and recall $j': P^1_XL \to E$ is the fiberwise linear map between vector bundles defined in \S\ref{sec; logrank}.

\begin{prop}
When $D$ is of linear Jacobian type, we have $j'^{-1}(M)=\overline{J\Lambda^\sharp}$ as schemes. 
\end{prop}

\begin{rema}
\label{rema: blowupJacobian}
We have $\overline{J\Lambda^\sharp}=\textup{Spec}(\textup{Rees}_{\mathscr{O}_X}(\mathcal{J}_D))$. Therefore $\mathbb{P}(\overline{J\Lambda^\sharp})$ agrees with the blowup of $X$ along the Jacobian ideal $\mathcal{J}_D$ for any reduced divisor $D$.
\end{rema}

Let $\{\Sigma_i\}$ be the set of irreducible components of $j'^{-1}(X)$, let $\pi:P^1_XL \to X$ be the projection and let $\pi_i: \Sigma_i \to X$ be the restrictions of $\pi$. Note that one of the irreducible components is $\overline{J\Lambda^\sharp}$ and all the others are mapped into $D_{sing}$ as in the free divisor case. In the spirit of the weak log transversality criterion for strong Euler homogeneity, we give a restatement of Proposition \ref{prop; logrankSEH}, generalizing \cite[Proposition 5.10]{LZ24-1}.

\begin{prop}
$D$ is strongly Euler homogoneous at $p\in D$ if and only if $\pi^{-1}_i(p) \subset T^*X\otimes L$ for every $i$.
\end{prop}

When $X$ is a surface, we can explicitly determine the cycle $[j'^{-1}(X)]$.

\begin{prop}
When $\dim X=2$ and $D$ is any reduced divisor on $X$, we have
\[
[j'^{-1}(X)]= [\overline{J\Lambda^\sharp}] + \sum_i (\mu(P_i)-\tau(P_i))[\pi^{-1}(P_i)],
\]
where $\mu(P_i)$ and $\tau(P_i)$ are the local Milnor number and Tjurina number at the singularity $P_i$.
\end{prop}

\begin{proof}
In this case $D$ is always free. We explained that the dimension of the irreducible components of $\overline{J\Lambda^\sharp} \setminus j'^{-1}(X)$ is no less than $2+1=3$ and those irreducible components are supported over $D_{sing}$. So $\overline{J\Lambda^\sharp} \setminus j'^{-1}(X) =\cup (\pi^{-1}(P_i))$ is the only possibility. To get the multiplicity for $\pi^{-1}(P_i)$, note that this number depends only on the singularity $P_i$, so we can shrink $X$ and assume $D$ has only one singularity $P$. We can compactify $X$ and $D$ without introducing new singularities on $D$ (by resolution of singularities), hence we may further assume $X$ is compact. Let $k: T^*X \to T^*X(\log D)$ and $k':P^1_XL\otimes L^\vee \to T^*X(\log D)$ be the untwisted logarithmic and augmentented logarithmic maps, i.e. $k=j\otimes L^\vee$ and $k'=j'\otimes L^\vee$. We have
\[
\int c_2(T^*X(\log D))= \sharp([X]\cdot [X]) = \sharp([X] \cdot k_*[X])= \sharp (k^*[X] \cdot [X])
\]
where the first intersection product $[X]\cdot [X]$ is calculated inside $T^*X(\log D)$. Let $i: T^*X \to P^1_XL\otimes L^\vee$ be the inclusion, we have
\[
k^*[{X}]=i^*k'^*[X]=i^*([\overline{\Lambda^\sharp}]+m\cdot[\pi^{-1}(P)])
\]
where $\overline{\Lambda^\sharp}=\overline{J\Lambda^\sharp}\otimes L^\vee$ is the untwisted global Ginzburg's sharp construction for $\Lambda=\textup{CC}(1_M)$ defined in \cite[Definition 3.12]{LZ24-1}. By Ginzburg's theorem \cite[Proposition 3.13]{LZ24-1}, we have
\[
i^*([\overline{\Lambda^\sharp}]+m\cdot[\pi^{-1}P])= \textup{CC}(1_U) + m\cdot \textup{CC}(1_P).
\]
Therefore
\[
\sharp (k^*[X] \cdot [X]) = \chi(U)+ m 
\]
by the global index theorem. Finally,
\[
m=\int c_2(T^*X(\log D)) -\chi (U) = \int c_2(TX(-\log D)) -\chi (U) = \mu(P) - \tau(P)
\]
by the proof of \cite[Theorem 3.1]{MR2928937} and \cite[Corollary 3.2]{MR2928937}.
\end{proof}

\begin{rema}
\label{remark:inverseLambdalog}
The property $k^{-1}(X)$ is a union of conormal spaces is interesting. Let $\Lambda=T^*_ZM$ where $Z$ is any subvariety of $X$ not contained in $D$. In general, without assuming $Z$ is log transverse to $D$, we do not know when $k^{-1}(\Lambda^{\log})$ is a union of conormal spaces (The definition of $\Lambda^{\log}$ is given at \cite[Definition 3.3]{LZ24-1}).
\end{rema}

\begin{lemm}
\label{lemma; SfZfinPn}
Let $X=\mathbb{P}^n$, and let $f\in H^0(\mathbb{P}^n,\mathcal{L})$ such that $f=0$ defines a reduced divisor $D$.
\begin{enumerate}
\item $S_f=\mathbb{P}(\overline{J\Lambda^\sharp})$.
\item $Z_f=\mathbb{P}(j'^{-1}(M))$.
\item $I=\mathbb{P}(T^*\mathbb{P}^n\otimes L)$ under the isomorphism $\mathbb{P}(P^1_{\mathbb{P}^n}L)\cong \mathbb{P}^n\times \mathbb{P}^n$.
\end{enumerate}
\end{lemm}

\begin{proof}
(i) follows from Remark \ref{rema: blowupJacobian}. Since $\mathcal{P}^1_{\mathbb{P}^n}\mathcal{L}\cong \sO(d-1)^{n+1}$ if $\deg \mathcal{L}=d$ by Proposition \ref{prop; decompofP1L}, $\mathbb{P}(P^1_{\mathbb{P}^n}L)\cong \mathbb{P}^n\times \mathbb{P}^n$ and the map $j':P^1L \to E$ is represented by the matrix $M^T_f$. Clearly the inverse image of zero section of $E$ is defined by the ideal $I_{Z_f}$ hence (ii) follows. To prove (iii), we use the notation in the proof of Proposition \ref{prop; decompofP1L}, and it is enough to show that $\Omega^1_{\mathbb{P}^n}$ is defined by $y_0x_0+ \ldots + y_nx_n=0$ in $\Omega^1_L(\log \mathbb{P}^n)\vert_{\mathbb{P}^n}$. On $U_i$, the expression $\sum r_j\otimes y_jx_j$ is trivialized into the expression
\[
\sum_{i=0}^n \Big(\frac{1}{k}\frac{dt_i}{t_i}\otimes \frac{y_jx_j}{x_i} - d(\frac{x_j}{x_i})\otimes y_j\Big).
\]
So when $\sum y_jx_j=0$, the undesirable differential form $\frac{dt_i}{t_i}$ disappears.
\end{proof}

We summarize below the avatar in $\mathbb{P}^n$ of the general properties listed above.

\begin{coro}
\label{coro; CC}
For any reduced divisor $D$ on $\mathbb{P}^n$, $S_f$ is always an irreducible component of $Z_f$. $D$ is strongly Euler homogeneous if and only if $Z_f \setminus S_f \subset I$, or equivalently $Z_f\vert_D \subset I$. If $D$ is of linear Jacobian type, then $S_f=Z_f$ as schemes. If $n=2$, then
\[
[Z_f] = [S_f] + \sum_i (\mu(P_i)-\tau(P_i))[P_i \times \mathbb{P}^2].
\]
\end{coro}

\section{Generalization to Smooth Projective Toric Varieties} 
\label{sec; toric}
In this section we prove a syzygy characterization for strong Euler homogeneity when $X$ is a smooth projective toric variety, generalizing Theorem~\ref{theo; syzygyrankcriterion}. 
\begin{setu}
\label{setup; toricvar}
$X=X(\Delta)$ is a smooth  $n$-dimensional projective toric variety determined by a fan $\Delta\subset N\cong \ZZ^n$.  We  assume that $\Delta(1)$ spans $N_\RR=N\otimes_\ZZ \RR$, where $\Delta(1)$ denotes the set of $1$-dimensional cones of $\Delta$. 
\end{setu}
Let $T=N\otimes_\ZZ \CC^*$ be the torus acting on $X$. 
Each $\rho\in \Delta(1)$ is a smooth cone and  corresponds to an irreducible $T$-invariant   divisor  $D_\rho$ of $X$. The Picard group $\textup{Pic}(X)\cong  A_{n-1}(X)$ is the free abelian group generated   by $\{\alpha_1=[D_{\rho_{1}}], \cdots ,\alpha_r=[D_{\rho_{r}}]\}$ for some $\rho_i\in \Delta(1)$,  where $r$ denotes the Picard number of $X$. We have the following  relation  $s:=|\Delta(1)|=r+n$.

We consider the polynomial ring $S=\CC[\{x_\rho|\rho\in \Delta(1)\}]$. Since each monomial $\underline{m}=\prod_\rho x_\rho^{a_\rho}$ determines a divisor $D=\sum_\rho a_\rho D_\rho$, 
we  denote  $\underline{m}$ by  $x^D$ and define 
$
\deg (x^D):=[D]\in \textup{Pic}(X) \/.
$
The $\textup{Pic}(X)$-grading on $S$ is then given by:
\[
S=\bigoplus_{\alpha\in \textup{Pic}(X)} S_\alpha \/, \    \text{ where }   S_\alpha:=\bigoplus_{\deg x^D=\alpha} \CC\cdot x^D\/.
\]
This $\textup{Pic}(X)$-graded ring $S$ is called the homogeneous coordinate ring of $X$ from \cite{Cox95}.   
We also define the shifted graded ring 
$
S(\alpha):=\bigoplus_{\beta\in \textup{Pic}(X)}  S_{\alpha+\beta} 
$ for any $\alpha\in \textup{Pic}(X)$.

For each cone $\sigma\in \Delta$ we define a monomial $x^{\hat{\sigma}}:=\prod_{\rho\notin \sigma} x_\rho$. These monomials generate the 
 irrelevant ideal  $B=(x^{\hat{\sigma}}:\sigma \in \Delta)$ of $S$.
 Let $V(B)$ be the subscheme of $\CC^{s}=\textup{Spec}(S)$  cut by  $B$. 
The torus $G:=\Hom(\textup{Pic}(X), \CC^*)\cong (\CC^*)^{r}$ acts on $\CC^{s}$  such that $U:=\CC^{s}\setminus V(B)$ is $G$-invariant. Since $X$ is assumed smooth, $G$ acts on $U$ freely and the toric map $\pi\colon U\to X=U//G$  is a geometric quotient (see \cite[Theorem 2.1]{Cox95} and  \cite[Lemma 5.1]{MR1841355}). In particular $U$ is a  $(\CC^*)^{r}$-fiber bundle over $X$.

In \cite{Cox95} (see also \cite[Chap 6, Appendix]{CLS11}) Cox proved a   local-global correspondence. 
\begin{theo}
\label{theo; toriccorrespondence}
There exist functors 
\[
\Gamma_*\colon Coh(X) \to GrMod_f(S) \text{ and }\  \widetilde{(\bullet)}\colon  GrMod_f(S)\to  Coh(X)
\]
between the category of coherent $\sO_X$ modules and the category of  finitely generated $\textup{Pic}(X)$-graded    $S$-modules such that   $\widetilde{(\bullet)}$ is exact and
$
\widetilde{ \Gamma_*(\cF)  }=\cF$ for any $\cF\in Coh(X) \/.$
\end{theo} 
The $\widetilde{(\bullet)}$ operation can be described as follows.  The space $U$ has a $G$-invariant affine open covering $U=\bigcup_{\sigma\in \Delta} U_\sigma$
where each $U_\sigma:=\CC^s\setminus V(x^{\hat{\sigma}})$ has  coordinate ring  $S_\sigma:=S_{x^{\hat{\sigma}}}$, the localization of $S$ along $x^{\hat{\sigma}}$. The action of $G$ on $U_\sigma$ induces an action on $S_\sigma$ and we consider the $G$-invariant subring 
$
i_\sigma\colon (S_\sigma)^G=(S_\sigma)_0\hookrightarrow S_\sigma \/.
$
Then $\pi|_{U_\sigma}$ is induced by $i_\sigma$  and $\{V_\sigma:=\textup{Spec}((S_\sigma)_0)=\pi(U_\sigma)\}$ forms an open affine  covering of $X$.  
For each $\textup{Pic}(X)$-graded $S$-module $M$, the base change of the localized module $M_\sigma$ along $i_\sigma$ gives a $(S_\sigma)_0$-module  and hence induces a coherent sheaf on 
$V_\sigma$. Gluing them together we then obtain the $\sO_X$-module $\widetilde{M}$.

\begin{exam} 
\label{exam; toric}
Let $D$ be a  divisor in $X$ such that $[D]=\alpha\in \textup{Pic}(X)$. Then we have
\[
H^0(X, \sO_X(\alpha))\cong S_\alpha \text{ and }\  \Gamma_*(\sO_X(\alpha))=S(\alpha) \/.
\]
Hence we have $\widetilde{S(\alpha)}=\sO_X(D)$.  
\end{exam}

We are now ready to state the generalization of our syzygy rank characterization. Let $X$ be a smooth projective toric variety as in Setup~\ref{setup; toricvar} and $D\subset X$ be a reduced hypersurface cut by   a global homogeneous polynomial $f\in H^0(X, \sO_X(D))=S_\alpha$, where $[D]=\alpha \neq 0\in \textup{Pic}(X)$.  
We denote  $\sO_X(D)$ by $\cL$. Let  $f_{x_\rho}:=\frac{\partial f}{\partial x_\rho}$ be the partial derivative of $f$ with respect to $x_\rho$, 
 we  have the generalized Euler relation:
\begin{equation}
\label{seq; generalizedEulerRelation}
\sum_{\rho\in \Delta(1)} \phi(D_\rho)\cdot x_\rho  \cdot f_{x_\rho} =\phi(\alpha)\cdot f \/.
\end{equation}
for any  $\phi\in \Hom_{\ZZ}(\textup{Pic}(X), \ZZ)$ (see \cite[Exercise 8.1.8]{CLS11}). 

Similar to \S\ref{sec; syzygyrank} we define  the global Jacobian ideal $J_f:=(\{f_{x_\rho}|\rho\in \Delta(1)\})$ of $S$. Then by the generalized Euler relation we have $f\in J_f$. 
We also define $\cJ_D$  to be   the ideal sheaf of the  singularity subscheme of $D$ and $\cR_D:=\cJ_D/\cI_D$, where $\cI_D=\widetilde{(f)}$ is the ideal sheaf of $D$.

\begin{prop} 
\label{prop; JDandcJD}
We have $\widetilde{J_f}=\cJ_D$  and consequently  $\widetilde{J_f/(f)}=\cR_D$. 
\end{prop}
\begin{proof}
Let $\hat{D}=V(f)\subset \CC^s$ be the `affine cone' of $D$, then the global Jacobian ideal $J_f$ defines the singularity subscheme of $\hat{D}$. 
By the $(\CC^*)^r$-fiber bundle structure, the associated ideal sheaf $\widetilde{J_f}$ defines the singular subscheme of $D$. 
The second statement follows from the exactness of the $\widetilde{(\bullet)}$ operation and the short exact sequence 
$
0 \longrightarrow (f) \longrightarrow J_f \longrightarrow J_f/(f) \longrightarrow 0 \/.
$
\end{proof} 
 
Following \cite[Proposition 8.18]{MS05} we may form an  exact sequence  of $\textup{Pic}(X)$-graded modules:
\begin{equation}
\label{seq; toricsyzygy1}
\begin{tikzcd}
  \bigoplus_{\alpha \in \textup{Pic}(X)} S(\alpha)^{\oplus a_\alpha} \arrow[r, "M'_f"] & \bigoplus_{\rho\in \Delta(1)}  S([D_\rho])\cdot e_\rho \arrow[r, "\xi'"] & J_f([D]) \arrow[r] & 0   
\end{tikzcd} \/.
\end{equation}
Here $\xi'(e_\rho)= f_{x_\rho}$ and $\{a_\alpha\}\subset \ZZ$.

Recall that $\{[D_{\rho_1}], [D_{\rho_2}], \cdots , [D_{\rho_r}]\}$ is a basis of $\textup{Pic}(X)$. Then $\Hom(\textup{Pic(X)}, \ZZ)$ has the dual basis by   $\{\phi_{\rho_1}, \phi_{\rho_2}, \cdots ,\phi_{\rho_k}\}$, where $\phi_{\rho_j}([D_{\rho_i}])=\delta^i_j$ for $j=1,2,\cdots ,r$. 
By the generalized Euler relation~\eqref{seq; generalizedEulerRelation}
we can also form the following exact sequence of free $\textup{Pic(X)}$-graded modules:
\begin{equation}
\label{seq; toricsyzygy2}
\begin{tikzcd}
  \bigoplus_{k=1}^r S\cdot e'_k \oplus
  \bigoplus_{\alpha \in \textup{Pic}(X)} S(\alpha)^{\oplus a_\alpha}  
 \arrow[r, "M_f"] & \bigoplus_{\rho\in \Delta(1)}  S([D_\rho])\cdot e_\rho \arrow[r, "\xi"] & \left(J_f/(f)\right)([D]) \arrow[r] & 0
\end{tikzcd} \/, 
\end{equation}
where $\xi(e_\rho)=\overline{f_{x_\rho}}$, $M_f(e'_k)=\sum_\rho \phi_{\rho_k}(D_\rho)  x_\rho  e_\rho$ and $M_f$ agrees with $M'_f$ on the second factor.

By Example~\ref{exam; toric} we have $\bigoplus_{\rho\in \Delta(1)}  \widetilde{S(-[D_\rho])}\cong \bigoplus_{\rho\in \Delta(1)} \sO_X(- D_\rho)$. We denote this vector bundle by $\Upsilon$ following the notation in \cite{MR1420708}. 
Then from \cite[Theorem 1.2]{MR1420708} we have  a
 commutative diagram with exact rows 
\begin{equation}
\label{diagram; toricsyzygy}
\begin{tikzcd} 
& & \cQ   \arrow[r, "="]& \cQ  & \\
0 \arrow[r] & \Omega_X^1 \arrow[r, "i_X"] & \Upsilon  \arrow[r, "\tilde{\sigma}"]
\arrow[u,    two heads , "\tilde{\kappa}"] & \textup{Pic}(X)\otimes_{\ZZ_X} \sO_X \arrow[r]\arrow[u,two heads , "\tilde{\kappa}'"] & 0 \\
0 \arrow[r] & \Omega_X^1  \arrow[r, "v"]  \arrow[u, "="] & \left(\cP^1_X \cL\right)\otimes \cL^\vee \arrow[r, " "] \arrow[u, "u"] &  \sO_X \arrow[r] \arrow[u, "\tilde{\tau}"] & 0  
\end{tikzcd} \/.
\end{equation} 
The morphism $\tilde{\sigma}$ is induced by $\sigma\colon \bigoplus_{\rho\in \Delta(1)} S(-[D_\rho]) \to \textup{Pic}(X)\otimes_{\ZZ} S$ that sends $(g_\rho)_{\rho\in \Delta(1)}$ to $\sum_{\rho} [D_\rho]\otimes x_\rho g_{\rho}$. The right column is induced from the   exact sequence of abelian groups
\[
\begin{tikzcd}
  \ZZ \arrow[r, " \tau "] & \textup{Pic}(X) \arrow[r, "\kappa'"] & Q_\ZZ \arrow[r, ""] & 0
\end{tikzcd} \/,
\]
where $\tau$ sends $1$ to $[D]$.

\begin{prop}
\label{prop; vbmor}
If $[D]\neq 0$ in $\textup{Pic}(X)$, then   the morphism  $\tau$ is injective and $Q:= Q_\ZZ\otimes_{\ZZ} S$ is free.  
Thus both $\tilde{\tau}$ and $u$ are inclusions of vector subbundles and $\cQ\cong \widetilde{Q} \cong \sO_X^{r-1}$ is trivial. 
\end{prop}
\begin{proof}
Recall that $\{[D_{\rho_1}],   \cdots , [D_{\rho_r}]\}$ is a basis of $\textup{Pic}(X)$ and $\{\phi_{\rho_1}, \cdots ,\phi_{\rho_r} \}$ 	is the dual basis of $\textup{Hom}(\textup{Pic}(X), \ZZ)$. 
At each component  $[D_{\rho_k}]\otimes S$ the morphism $\tau$  is  the multiplication by degree  $\phi_{\rho_k}(D)$. Since $[D]\neq 0$, at least some $\phi_{\rho_k}(D)\neq 0$ and hence $\tau$ is injective.  
\end{proof}

Abusing notation we also denote by $\kappa'$ the induced map $\textup{Pic}X\otimes S\to Q$. Let   $\kappa$ the composition map $\kappa'\circ \sigma$ and   $K=\ker \kappa$ be the kernel module of $\kappa$.  
\begin{prop}
 If $[D]\neq 0\in \textup{Pic}(X)$,  then we have 
$\cP_X^1\sO_X(D)\cong \widetilde{K([D])}\cong \widetilde{K}\otimes \sO_X(D)$. 
\end{prop}

Since $M'_f$ and $M_f$ are $\textup{Pic}(X)$-graded,  the  ranks $\rk M'_f(P)$ and $\rk M_f(P)$ are constant  for  any $P\in \pi^{-1}(p)$. Thus we  denote them by $\rk M'_f(p)$ and $\rk M_f(p)$ respectively.   We  denote
by $\kappa$ the composition map $\kappa'\circ \sigma$.  
Then the dimension of the vector space 
 $\kappa(P)\big(\ker M'^\vee_f(P)\big)$ is   independent of  $P\in \pi^{-1}(p)$ and equals the dimension of 
 $\tilde{\kappa}(p)\big(\ker \widetilde{M'^\vee_f}(p)\big)\subset \cQ(p)$.
\begin{defi}
\label{defi; logdefect}
We call $\rk M'_f(p)$ and $\rk M_f(p)$ the first syzygy rank  and the augmented first syzygy rank of $D$ at $p$.  
We call $\dim \left(\tilde{\kappa}(p)\big(\ker \widetilde{M'^\vee_f}(p)\big)\right)$  the logarithmic defect of the pair $(X, D)$ at $p$ and denote it by $\textup{Def}_{X,f}(p)$. 
\end{defi}

Theorem~\ref{theo; syzygyrankcriterion}  generalizes as follows.
\begin{theo}
\label{theo; toricsyzygy}
Let $X$ be a smooth projective toric variety as in Setup~\ref{setup; toricvar}.  Let $D\subset X$ be a reduced hypersurface cut by a global homogeneous polynomial $f\in S_{[D]}$  such that $[D]\neq 0$. Then $D$ is strongly Euler homogeneous at $p$ if and only if 
$$
\rk M'_f(p)+\textup{Def}_{X,f}(p)  = \rk M_f(p) \/.
$$
\end{theo}

\begin{rema}   
When $X=\PP^n$,   $u$ is an isomorphism for   any hypersurface $D$  of degree $\geq 1$ and  $\cQ(p)=\{0\}$. Thus we have $\textup{Def}_{X,f}(p)=0$ and the above theorem recovers  Theorem~\ref{theo; syzygyrankcriterion}.
\end{rema}

\begin{proof}[Proof of Theorem~\ref{theo; toricsyzygy}] 
The surjective graded  morphisms $\xi'$ and $\xi$ in \eqref{seq; toricsyzygy1} \eqref{seq; toricsyzygy2}  induce surjective  morphisms 
\[
\widetilde{\xi'}\colon \Upsilon^\vee  \longrightarrow \cJ_D\otimes \cL
  \ \text{ and } \ 
\widetilde{\xi}\colon \Upsilon^\vee \longrightarrow \cR_D\otimes \cL  \/.
\] 
A local computation shows that $\widetilde{\xi'}$ factors through $u^\vee$ and $\rho'$ in sequence \eqref{seq; sesq2global} 
\[
\widetilde{\xi'}=\rho'\circ u^\vee \colon \Upsilon^\vee \longrightarrow \left(\cP^1_X \cL\right)^\vee\otimes \cL \longrightarrow \cJ_D\otimes \cL \/,
\]  
while $\widetilde{\xi}$ factors through $i_X^\vee=v^\vee\circ u^\vee$ and $\rho$ in sequence \eqref{seq; sesq1global}
\[
\widetilde{\xi}=\rho\circ i_X^\vee \colon \Upsilon^\vee \longrightarrow \left(\cP^1_X \cL\right)^\vee\otimes \cL \longrightarrow  \textup{Der}_X \longrightarrow \cR_D\otimes \cL \/.
\]

We then have the following diagram with exact rows
\[
\begin{tikzcd} 
0\arrow[r] & Im \left(\widetilde{M'_f} \right) \arrow[r]\arrow[d, " "] & \Upsilon^\vee \arrow[r, "\widetilde{\xi'}"]\arrow[d, "u^\vee"]  & \cJ_D\otimes \cL \arrow[r] \arrow[d, "="]  & 0 \\
0\arrow[r] & \textup{Der}_X(-\log D) \arrow[r, ""] \arrow[d, "="]  & \left(\cP^1_X \cL\right)^\vee\otimes \cL \arrow[r, "\rho'"] \arrow[d,  "v^\vee "]    & \cJ_D\otimes \cL \arrow[r] \arrow[d, two heads," "]  & 0  \\
0\arrow[r] & \textup{Der}_X(-\log D) \arrow[r, ""] & \textup{Der}_X \arrow[r, "\rho'"]    & \cR_D\otimes \cL \arrow[r]  & 0  
\end{tikzcd} 
\]   
By  Proposition~\ref{prop; vbmor} we see that   $u^\vee$ is surjective. Then $u^\vee$ maps $Im \left(\widetilde{M'_f} \right)$ onto $\textup{Der}_X(-\log D)$ and we may    take $\sP_1$ to be
$\bigoplus_{\alpha \in \textup{Pic}(X)} \sO(\alpha)^{\oplus a_\alpha}$ when computing the morphisms $j$ and $j'$. Thus we have 
\[
\begin{tikzcd}
 \bigoplus_{\alpha \in \textup{Pic}(X)} \sO(\alpha)^{\oplus a_\alpha} \arrow[r, "\widetilde{M'_f}"]\arrow[d, "="] & \Upsilon^\vee \arrow[r, "\widetilde{\xi'}"]\arrow[d, two heads , "u^\vee"]  & \cJ_D\otimes \cL \arrow[r] \arrow[d, "="]  & 0 \\
   \bigoplus_{\alpha \in \textup{Pic}(X)} \sO(\alpha)^{\oplus a_\alpha}  \arrow[r, "j'^{\vee}"] & \left(\cP^1_X \cL\right)^\vee\otimes \cL \arrow[r, "\rho'"]    & \cJ_D\otimes \cL \arrow[r]  & 0  
\end{tikzcd}  \/.
\]
Combining with Diagram~\eqref{diagram; toricsyzygy},   evaluating at $p$  we have
\[
\begin{tikzcd}
  \left(\cP^1_X \cL \otimes \cL^\vee \right)(p)\cong \CC^{n+1} \arrow[r, hook,   "u(p)"]\arrow[d, "j'(p)"'] & 
  \Upsilon(p)\cong \CC^s  \arrow[r, two heads , "\tilde{\kappa}(p)"] \arrow[dl, "\widetilde{M'^\vee_f}(p)"] & \cQ(p)\cong \CC^s/Im(u(p)) \\
\CC^{m}  &  &  
\end{tikzcd} \/.
\]
Here $m:= \sum_\alpha a_\alpha$. Since $\tilde{\kappa}(p)$ is surjective and   $\ker \tilde{\kappa}(p)= Im(u(p))$, we have
\[
\dim \ker \widetilde{M'^\vee_f}(p) =\dim \left( Im(u(p))\cap \ker \widetilde{M'^\vee_f}(p) \right) + \dim \tilde{\kappa}(p)\left( \ker \widetilde{M'^\vee_f}(p) \right) \/.
\] 
Since $u(p)$ is injective, we have
\[
\dim \left( Im(u(p))\cap \ker \widetilde{M'^\vee_f}(p) \right)=\dim \ker\left(\widetilde{M'^\vee_f}(p)\circ u(p) \right)=\dim \ker j'(p) \/.
\]
As $\rk \widetilde{M'_f}(p)=\rk \widetilde{M'^\vee_f}(p)$,  we have
\[ 
\rk \widetilde{M'_f}(p) = s-\dim \ker \widetilde{M'^\vee_f}(p)= s-  \dim \ker j'(p)
 - \text{Def}_{X,f}(p)   
=    \rk j'(p) + r -1- \text{Def}_{X,f}(p)  \/.
\]

On the other hand, identifying $\bigoplus_{k=1}^r \widetilde{S\cdot e'_k}$ with $\textup{Pic}(X)\otimes_{\ZZ_X} \sO_X$ we have   
\[
\begin{tikzcd} 
\textup{Pic}(X)\otimes_{\ZZ_X} \sO_X   \arrow[d, hook] \arrow[r, "="] & \textup{Pic}(X)\otimes_{\ZZ_X} \sO_X 
\arrow[d, hook]  & &\\
\textup{Pic}(X)\otimes_{\ZZ_X} \sO_X  \oplus
 \bigoplus_{\alpha \in \textup{Pic}(X)} \sO(\alpha)^{\oplus a_\alpha}
 \arrow[r, "\widetilde{M_f} "]\arrow[d,  two heads, "pr_2"] & \Upsilon^\vee \arrow[r, "\widetilde{\xi}"]
 \arrow[d, two heads, "i_X^\vee"]  & \cR_D\otimes \cL \arrow[r] \arrow[d, "="]  & 0 \\
  \bigoplus_{\alpha \in \textup{Pic}(X)} \sO(\alpha)^{\oplus a_\alpha} \arrow[r, "j^{\vee}"]   & \textup{Der}_X \arrow[r, "\rho'"]   & \cR_D\otimes \cL  \arrow[r] & 0  
\end{tikzcd}   \/.
\]     
This shows that  
\[
\rk \widetilde{M_f}(p) =\rk j(p)+r 
=\rk \widetilde{M'_f}(p) +\textup{Def}_{X,f}(p)+\rk j(p)+1-\rk j'(p)  \/.
\]
Since $\rk \widetilde{M'_f}(p)=\rk M'_f(p)$ and $\rk \widetilde{M_f}(p)=\rk M_f(p)$, applying  Proposition~\ref{prop; logrankSEH} we then obtain the statement of the theorem.
\end{proof}

From the above proof   we  see that 
$
\rk  M'_f(p)\leq \rk j'(p) +r-1\leq \rk j(p)+r = \rk M_f(p)  
$
always holds. Thus we have the following sufficient criterion.
 \begin{coro}
Under the same assumption as in Theorem~\ref{theo; toricsyzygy}, if $\rk  M'_f(p)=\rk  M_f(p)$  then $D$ is strongly Euler homogeneous at $p$. 
 \end{coro} 
 
If $D$ has an isolated singularity at $p$, then $\rk j(p)=0$ and hence $\rk  M_f(p) =\rk \textup{Pic}(X)$ equals the Picard number of $X$. Thus we have 
\begin{coro}
Under the same assumption as in Theorem~\ref{theo; toricsyzygy}. If $p$ is an isolated singularity, then  $D$ is strongly Euler homogeneous at $p$ if and only if $\rk  M'_f(p)+\textup{Def}_{X,f}(p)=\rk \textup{Pic}(X)$. 
\end{coro}

\newpage 
 
\bibliographystyle{plain}
\bibliography{refRQHC}

\end{document}